\documentclass[a4paper,article]{amsart}
\usepackage{amsfonts}
\usepackage{amssymb}
\usepackage{amsthm}
\usepackage{enumerate}
\usepackage[english]{babel}
\usepackage{hyperref}
\usepackage[T1]{fontenc}
\usepackage[utf8]{inputenc}
\usepackage{caption}
\usepackage{mathtools}
\usepackage[usenames,dvipsnames]{pstricks}

\usepackage{color}

\theoremstyle{plain}
\begingroup
\theoremstyle{plain}
\newtheorem{theorem}{Theorem}[section]
\newtheorem{corollary}[theorem]{Corollary}
\newtheorem{proposition}[theorem]{Proposition}
\newtheorem{lemma}[theorem]{Lemma}
\theoremstyle{definition}
\newtheorem{definition}[theorem]{Definition}

\theoremstyle{remark}
\newtheorem{remark}[theorem]{Remark}
\theoremstyle{claim}
\newtheorem{claim}[theorem]{Claim}
\theoremstyle{definition}
\newtheorem{notation}[theorem]{Notation}
\endgroup

\theoremstyle{definition}
\theoremstyle{remark}

\mathchardef\emptyset="001F

\numberwithin{equation}{section}

\makeatletter
\def\Ddots{\mathinner{\mkern1mu\raise\p@
\vbox{\kern7\p@\hbox{.}}\mkern2mu
\raise4\p@\hbox{.}\mkern2mu\raise7\p@\hbox{.}\mkern1mu}}
\makeatother

\title[]
{Switching in time-optimal problem \\with control in a ball.}
\author[A. A. Agrachev]{Andrei A. Agrachev}
\address[A. A. Agrachev]{SISSA, 34136 Trieste,  Italy; Steklov Mathematical Institute, 119991 Moscow, Russia}
\email[A. Agrachev]{agrachev@sissa.it}
\author[C. Biolo]{Carolina Biolo}
\address[Carolina Biolo]{SISSA, Via Bonomea 265, 34136 Trieste, Italy}
\email[Carolina Biolo]{cbiolo@sissa.it}
\date{}

\begin{document}

\begin{abstract}
In this paper we analyse local regularity of time-optimal controls and trajectories for an $n$-dimensional affine control system with a control parameter, taking values in a $k$-dimensional closed ball. In the case of $k=n-1$, we give sufficient conditions in terms of Lie bracket relations for all optimal controls to be smooth or to have only isolated jump discontinuities.
\end{abstract}
\maketitle
\tableofcontents
\section{Introduction}

In this paper, we continue to study singularities of the extremals of the time-optimal problem for the control system of the form:
$$
\dot q=f_0(q)+\sum_{i=1}^ku_if_i(q),\quad q\in M,\ (u_1,\ldots,u_k)\in U, \eqno (1.1)
$$
where $M$ is a smooth $n$-dimensional manifold, $U=\{u\in \mathbb{R}^k : ||u||\leq 1\}$ is the $k$-dimensional ball, and $f_0,\,f_1,\,\ldots,\,f_k$ are smooth\footnote{We work in $\mathcal{C}^\infty (M)$ category but all results are true for $\mathcal{C}^2(M)$ vector fields.} vector fields. We also assume that $f_1(q),\ldots,f_k(q)$ are linearly independent in the domain under consideration.

If $k=n$, then all extremals are smooth; otherwise they may be nonsmooth and there exists a vast literature dedicated to the case $k=1$. Some references can be found in paper \cite{AB}, where we studied the simplest intermediate case $k=2,\ n=3$. It appears that the developed in \cite{AB} techniques work in much more general setting than we expected and can be efficiently applied to any pair $k<n$ giving a clear explicit description of less degenerate singularities (see Theorem \ref{resultneq} of the current paper).

Moreover, if $k=n-1,\ q\in M$, and $f_0,f_1,\ldots,f_{n-1}$ is a generic germ of $n$-tuple of vector fields at $q$, then the germs of extremal at $q$ may have only these less degenerate singularities. More precisely, let us define a vector $a\in\mathbb R^{n-1}$ and a matrix $A\in\mathrm{so}(n-1)$ by the formulas:
$$
a(q)=\{\det\left(f_1(q),\ldots,f_{n-1}(q),[f_0,f_i](q)\right)\}_{i=1}^{n-1},
$$
$$
A(q)=\{\det\left(f_1(q),\ldots,f_{n-1}(q),[f_i,f_j](q)\right)\}_{i,j=1}^{n-1},
$$
where $[\cdot,\cdot]$ is a Lie bracket. We have the following:

\begin{theorem}
\label{codim1thm}
 If
\begin{equation}
\label{codim1}
a(\bar{q})\notin A(\bar{q})S^{n-2},
\end{equation} then there exists a neighbourhood $O_{\bar{q}}$ of $\bar{q}$ in $M$ such that any time-optimal trajectory contained in $O_{\bar{q}}$ is piecewise smooth with no more than 1 non smoothness point.
\end{theorem}
Here $S^{n-2}=\{u\in\mathbb R^{n-1}: ||u||=1\}$ is the unit sphere.

If $n=3,\ k=2$, then inequality (\ref{codim1}) reads:
\begin{equation}
\label{case32}
{\det}^2\left(f_1(\bar{q}),f_2(\bar{q}),[f_0,f_1](\bar{q})\right)+
{\det}^2\left(f_1(\bar{q}),f_2(\bar{q}),[f_0,f_2](\bar{q})\right)\neq
{\det}^2\left(f_1(\bar{q}),f_2(\bar{q}),[f_1,f_2](\bar{q})\right).
\end{equation}
In this case, the result of Theorem 1.1 follows from \cite[Th.\,3.1]{AB}, but the cited result of \cite{AB} is a bit  stronger than this. Indeed, assumption (\ref{case32}) is more restrictive than the used in \cite[Th.\,3.1]{AB} assumption
$$
\mbox{rank}{\{f_1(\bar{q}),f_2(\bar{q}),[f_0,f_1](\bar{q}),[f_0,f_2](\bar{q}),[f_1,f_2](\bar{q})\}}=3.
$$

In the next section we recall necessary background from the optimal control theory: the Pontryagin maximum principle and the Goh condition. Theorem~1.1 is a corollary of the main result stated in Section \ref{Statement} and proved in Section \ref{Proof}. The proof is based on the blow-up techniques and the structure of partially hyperbolic equilibria.

\section{Preliminaries} \label{sec:preliminaries.section}

In this section we recall some basic definitions in Geometric Control Theory. For a more detailed introduction, see \cite{A}.
\begin{definition}
Given a $n$-dimensional manifold $M$, we call $\mathrm{Vec}(M)$ the \emph{set of smooth vector fields} on $M$: $f\in \mathrm{Vec}(M)$ if and only if $f$ is a smooth map with respect to $q\in M$ taking value in the tangent bundle, $$f:M\longrightarrow TM,
$$ such that if $q\in M$ then $f(q)\in T_q M$.\\
Each vector field defines a \emph{dynamical system} $$\dot{q}=f(q),$$
i. e. for each initial point $q_0\in M$ it admits a solution $q(t,q_0)$ on an opportune time interval $I$, such that $q(0,q_0)=q_0$ and
$$\frac{d}{dt}q(t)=f(q(t)),\quad \mathrm{a.}\,\mathrm{e}. \,t\in I .
$$
\end{definition}

\begin{definition}
$f\in \mathrm{Vec}(M)$ is a \emph{complete vector field} if , for each initial point $q_0\in M$, the solution $q(t,q_0)$ of the \emph{dynamical system} $\dot{q}=f(q)$ is defined for every $t\in \mathbb{R}$. If $f\in \mathrm{Vec}(M)$ has a compact support, it is a complete vector field.
\end{definition}

In our local study, we may assume without lack of generality that all vector fields under consideration are complete.

\begin{definition}
A \emph{control system} in $M$ is a family of dynamical systems
$$\dot{q}=f_u(q), \quad \mathrm{with}\,\,q\in M,\, \{f_u\}_{u\in U}\subseteq \mathrm{Vec}(M),
$$
parametrized by $u\in U\subseteq\mathbb{R}^k$, called \emph{space of control parameters}.\\
Instead of constant values $u\in U$, we are going to consider $L^\infty$ time depending functions taking values in $U$. Thus, we call $\mathcal{U}=\{u:I\rightarrow U,\,u\in L^\infty\}$ the \emph{set of admissible controls} and study the following control system
\begin{equation}
\label{control.system}
\dot{q}=f_u(q), \quad \mathrm{with}\,\,q\in M,\, u\in \mathcal{U}.
\end{equation}
\end{definition}
With the following theorem we want to show that, choosing an admissible control, it is guaranteed the locally existence and uniqueness of the solution of a control system for every initial point.
\begin{theorem}
Fixed an admissible control $u\in\mathcal{U}$, (\ref{control.system}) is a non-autonomous ordinary differential equation, where the right-hand side is smooth with respect to $q$, and measurable essentially bounded with respect to $t$, then, for each $q_0\in M$, there exists a local unique solution $q_u(t,q_0)$ such that $q_u(0,q_0)=q_0$ and it is lipschitzian with respect to $t$.
\end{theorem}
\begin{definition}
\label{007}
We denote
$$A_{q_0}=\{q_u(t,q_0) : t\geq 0,u\in\mathcal{U} \}
$$
the \emph{attainable set} from $q_0$. \\
We will write $q_u(t)=q_u(t,q_0)$ if we do not need to stress  that the initial position is $q_0$.
\end{definition}
\begin{definition}
An\emph{ affine control system} is a control system of the following form
\begin{equation}
\label{affine.control.system}
\dot{q}=f_0(q)+\sum^k_{i=1}u_if_i(q), \quad q\in M
\end{equation}
where $f_0,\ldots, f_k$ $\in \mathrm{Vec}(M)$ and $(u_1,\ldots, u_k)\in \mathcal{U}$, taking values in the set $U\subseteq \mathbb{R}^k$. \\
The uncontrollable term $f_0$ is called \emph{drift}. \\
\end{definition}
\subsection{Time-optimal problem}
\begin{definition}
Given the control system (\ref{control.system}), $q_0\in M$ and $q_1\in A_{q_0}$, the \emph{time-optimal problem} consists in minimizing the time of motion from $q_0$ to $q_1$ via admissible trajectories:
\begin{equation}
\label{time-optimal.problem}\left\lbrace \begin{array}{ll}
\dot{q}=f_u(q)&u\in \mathcal{U}\\
q_u(0,q_0)=q_0&\\
q_u(t_1,q_0)=q_1&\\
t_1 \rightarrow \min&
\end{array} \right.
\end{equation}
We call these minimizer trajectories \emph{time-optimal trajectories}, and \emph{time-optimal controls} the corresponding controls.
\end{definition}
\subsubsection{Existence of time-optimal trajectories}
Classical Filippov's Theorem (See \cite{A}) guarantees the existence of a time-optimal control for the affine control system if $U$ is a convex compact and $q_0$ is sufficiently close to $q_1$.
\subsection{First and second order necessary optimality condition}
Now we are going to introduce basic notions about Lie brackets, Hamiltonian systems and Poisson brackets, so that we present the first and second order necessary conditions of optimality: Pontryagin Maximum Principle, and Goh condition.
\begin{definition}
Let $f,g\in \mathrm{Vec}(M)$, we define their \emph{Lie brackets} the following vector field
$$[f,g](q)=\frac{1}{2}\left.\frac{\partial^2}{\partial t^2}\right|_{t=0}e^{-t g}\circ e^{-t f}\circ e^{t g}\circ e^{t f}(q), \quad  \forall q\in M.
$$
where $e^{-t f}$ is the flow defined by $-f$.\\
\begin{center}
\psscalebox{1.0 1.0} 
{
\begin{pspicture}(0,-1.4669921)(6.62,1.4669921)
\psbezier[linecolor=black, linewidth=0.04](1.66,-0.73105466)(2.52,-1.1510547)(4.44,-0.47105467)(5.08,-0.7310546875)
\psbezier[linecolor=black, linewidth=0.04](5.1096,-0.72670466)(5.165974,-0.28143632)(5.3937507,0.19085746)(5.7704,0.2845953035607818)
\psbezier[linecolor=black, linewidth=0.04](5.74,0.24894531)(5.06,1.5289453)(3.26,0.24894531)(2.66,0.6689453125)
\psbezier[linecolor=black, linewidth=0.04](2.68,0.6689453)(2.5118368,-0.043554686)(1.0338775,0.071445316)(1.06,-0.0910546875)
\psline[linecolor=black, linewidth=0.04, arrowsize=0.05291667cm 2.0,arrowlength=1.4,arrowinset=0.12]{->}(1.68,-0.7510547)(1.12,-0.2710547)
\psdots[linecolor=black, dotsize=0.18](1.68,-0.7510547)
\psdots[linecolor=black, dotsize=0.18](1.1,-0.01)
\rput[bl](2.56,-1.3910546){$e^{tf}$}
\rput[bl](5.6,-0.3910547){$e^{tg}$}
\rput[bl](4.2,1.2089453){$e^{-tf}$}
\rput[bl](1.3,0.46894532){$e^{-tg}$}
\rput[bl](1.36,-1.3510547){$q$}
\rput[bl](0.0,-0.8110547){$[f,g](q)$}
\rput[bl](-2.8,0.0){$e^{-t g}\circ e^{-t f}\circ e^{t g}\circ e^{t f}(q)$}
\end{pspicture}
}
\captionof{figure}{Lie Bracket}
\end{center}
\end{definition}
\begin{definition}
An \emph{Hamiltonian} is a smooth function on the cotangent bundle
$$h\in C^\infty(T^*M).
$$
The \emph{Hamiltonian vector field} is the vector field associated with $h$ via the canonical symplectic form $\sigma$
$$\sigma_\lambda (\cdot ,\overrightarrow{h})=d_\lambda h.
$$
We denote
$$\dot{\lambda}=\overrightarrow{h}(\lambda), \quad  \lambda \in T^*M,
$$
the \emph{Hamiltonian system}, which corresponds to $h$.\\
Let $(x_1,\ldots,x_n)$ be local coordinates in $M$ and $(\xi_1,\ldots,\xi_n,x_1,\ldots,x_n)$ induced coordinates in $T^*M,\ \lambda=\sum_{i=1}^n\xi_idx_i$. The \emph{symplectic form} has expression $\sigma=\sum^n_{i=1}d\xi_i\wedge dx_i$. Thus, in canonical coordinates, the Hamiltonian vector field has the following form
$$\overrightarrow{h}=\sum^n_{i=1}\left( \frac{\partial h}{\partial \xi_i}\frac{\partial}{\partial x_i} -\frac{\partial h}{\partial x_i}\frac{\partial}{\partial \xi_i}\right).
$$
Therefore, in canonical coordinates, it is
$$\left\lbrace
\begin{array}{l}
\dot{x}_i=\frac{\partial h}{\partial \xi_i}\\
\dot{\xi_i}=-\frac{\partial h}{\partial x_i}
\end{array}
\right.
$$
for $i=1,\ldots,n$.
\end{definition}
\begin{definition}
The \emph{Poisson brackets} $\{a,b\}\in \mathcal{C}^\infty(T^*M)$ of two Hamiltonians $a,b\in \mathcal{C}^\infty(T^*M)$ are defined as follows: $\{a,b\}=\sigma(\vec a,\vec b)$; the coordinate expression is:
$$\{a,b\}=\sum_{k=1}^n\left( \frac{\partial a}{\partial \xi_k}\frac{\partial b}{\partial x_k}-\frac{\partial a}{\partial x_k}\frac{\partial b}{\partial \xi_k}\right).
$$
\end{definition}
\begin{remark}
\label{poisson,lie}
Let us recall that, given $g_1$ and $g_2$ vector fields in $M$, considering the Hamiltonians $a_1(\xi,x)=\left\langle \xi, g_1(x)\right\rangle $ and $a_2(\xi,x)=\left\langle \xi, g_2(x)\right\rangle $, it holds
$$\{a_1,a_2\}(\xi,x)=\left\langle \xi, [g_1,g_2](x)\right\rangle.
$$
\end{remark}
\begin{remark}
\label{derivPoiss}
Given a smooth function $\Phi$ in $\mathcal{C}^\infty(T^*M)$, and $\lambda(t)$ solution of the Hamiltonian system $\dot{\lambda}=\overrightarrow{h}(\lambda)$, the derivative of $\Phi(\lambda(t))$ with respect to $t$ is the following
$$\frac{d}{dt}\Phi(\lambda(t))=\{h,\Phi\}(\lambda(t)).
$$
\end{remark}
\subsubsection{Pontryagin Maximum Principle}
\begin{theorem}[Pontryagin Maximum Principle - time-optimal problem]
Let an admissible control $\tilde{u}$, defined in the interval $t\in [0,\tau_1 ]$, be time-optimal for the system (\ref{control.system}), and let the Hamiltonian associated with this control system be the action on $f_u(q)\in T^*_q M$ of a covector $\lambda\in T^*_q M$: $$\mathcal{H}_u(\lambda)=\left\langle \lambda,f_u(q)\right\rangle . $$
Then there exists $\lambda(t)\in T_{q_{\tilde{u}}(t)}^*M$, for $t\in [0,\tau_1 ]$, called \emph{extremal} never null and lipschitzian, such that for almost all $t\in [0,\tau_1 ]$ the following conditions hold:
\begin{enumerate}
	\item $\dot{\lambda}(t)=\vec{\mathcal{H}}_{\tilde{u}}( \lambda(t))$
	\item $\mathcal{H}_{\tilde{u}}(\lambda(t))= \max_{u\in U} \mathcal{H}_u(\lambda(t))$ (Maximality condition)
\item $\mathcal{H}_{\tilde{u}}(\lambda(t))\geq0$.
\end{enumerate}
Given the canonical projection $\pi:TM\rightarrow M$, we denote $q(t)=\pi(\lambda(t))$ the \emph{extremal trajectory}.
\end{theorem}
\subsubsection{Goh condition}
\label{subsecgoh}
Finally, we present the Goh condition, on the singular arcs of the extremal trajectory, in which we do not have information from the maximality condition of the Pontryagin Maxinum Principle. We state the Goh condition only for affine control systems (\ref{affine.control.system}).
\begin{theorem}[\emph{Goh condition}]
\label{Goh.condition}
Let $\tilde q(t),\ t\in[0,t_1]$ be a time-optimal trajectory corresponding to a control $\tilde u$. If $\tilde u(t)\in\mathrm{int}U$ for any $t\in(\tau_1,\tau_2)$,
then there exist an extremal $\lambda(t)\in T_{q(t)}^*M$ such that
\begin{equation}
\label{cond.di.Goh}
\left\langle \lambda(t),[f_i,f_j](q(t))\right\rangle =0,\quad\ t\in(\tau_1,\tau_2),\ i,j=1,\ldots,m.
\end{equation}
\end{theorem}

\subsection{Consequence of the optimality conditions.}
In this paper we are going to investigate the local regularity of time-optimal trajectories for the $n$-dimensional affine control system with a $k$-dimensional control:
\begin{equation}
\label{246}
\dot{q}=f_0(q)+\sum^k_{i=1}u_if_i(q),\quad q\in M, u\in\mathcal{U}
\end{equation}
where the space of control parameters is the $k$-dimensional closed unitary ball: $U=\{u\in\mathbb{R}^k : ||u||\leq 1\}$.\\
By the Pontryagin Maximum Principle, every time-optimal trajectory of our system has an extremal in the cotangent bundle $T^*M$ that satisfies a Hamiltonian system, given by the maximized Hamiltonian.
\begin{notation}
\label{notaz}
Let us call $h_i(\lambda)=\left\langle \lambda,f_i(q) \right\rangle $, $f_{ij}(q)=[f_i,f_j](q),\ f_{ijk}(q)=[f_i,[f_j,f_k]](q)$, $h_{ij}(\lambda)=\left\langle \lambda,f_{ij}(q) \right\rangle $, and $h_{ijk}(\lambda)=\left\langle \lambda,f_{ijk}(q) \right\rangle$, with $\lambda\in T^*_{q}M$ and $i,j,k\in\{0,1,\hdots, k\}$.\\
Moreover, we denote the following vector $H_{0I}(\lambda)=\{h_{0i}(\lambda)\}_{i}\in \mathbb{R}^k$ and $k\times k$ matrix $H_{IJ}(\lambda)=\{h_{ij}(\lambda)\}_{ij}$ with respect to $\lambda\in T^*M$.
\end{notation}
\begin{definition}
The \emph{singular locus} $\Lambda \subseteq T^*M$, is defined as follows:
$$
\Lambda=\{\lambda\in T^*M : h_1(\lambda)=\hdots=h_k(\lambda)=0\}.
$$
\end{definition}
The following proposition is an immediate corollary of the Pontryagin Maximum Principle.

\begin{proposition}
\label{776}
If an extremal $\lambda(t),\ t\in[0,t_1]$, does not intersect the singular locus $\Lambda$, then $\forall t\in [0,t_1]$
\begin{equation}
\label{time-opt.control}
\tilde{u}(t)=\left( \begin{array}{c}
\frac{h_1(\lambda(t))}{(h^2_1(\lambda(t))+\hdots+h^2_k(\lambda(t)))^{1/2}}\\
\vdots
\\
 \frac{h_k(\lambda(t))}{(h^2_1(\lambda(t))+\hdots+h^2_k(\lambda(t)))^{1/2}}
 \end{array}
 \right).
\end{equation}
Moreover, this extremal is a solutions of the Hamiltonian system defined by the Hamiltonian $\mathcal{H}(\lambda)=h_0(\lambda)+\sqrt{h^2_1(\lambda)+\hdots+h^2_k(\lambda)}$. Thus, it is smooth.
\end{proposition}

\begin{definition}
We will call \emph{bang arc} any smooth arc of a time-optimal trajectory $q(t)$, whose corresponding time-optimal control $\tilde{u}$ lies in the boundary of the space of control parameters: $\tilde{u}(t)\in \partial U$.
\end{definition}
\begin{corollary}
\label{017}
An arc of a time-optimal trajectory, whose extremal is out of the singular locus, is a bang arc.
\end{corollary}
From Corollary \ref{017} we already have an answer about the regularity of time-optimal trajectories:  every time-optimal trajectory, whose extremal lies out of the singular locus, is smooth.\\
However, we do not know what happen if an extremal touches the singular locus, optimal controls may be not always smooth.
\begin{definition}
A \emph{switching} is a discontinuity of an optimal control.\\
Given $u(t)$ an optimal control, $\bar{t}$ is a \emph{switching time} if $u(t)$ is discontinuous at $\bar{t}$.\\ Moreover given $q_u(t)$ the admissible trajectory, $\bar{q}=q_u(\bar{t})$ is a \emph{switching point }if $\bar{t}$ is a switching time for $u(t)$.
\end{definition}

A concatenation of bang arcs is called \emph{bang-bang trajectory}.

An arc of an optimal trajectory that admits an extremal totally contained in the singular locus $\Lambda$, is called \emph{singular arc}.

\section{Statement of the result}
\label{Statement}
Let us assume that $\dim M=n$ and study the time-optimal problem for the following system
\begin{equation}
\label{111}
\dot{q}=f_0(q)+\sum^k_{i=1}u_if_i(q),\quad q\in M, \,u\in\mathcal{U},
\end{equation}
where $k<n$, $f_0,f_1,\hdots,f_k$ are smooth vector fields, and $U=\{u\in\mathbb{R}^k: ||u||\leq1\}$; we also assume that $f_1,\hdots,f_k$ are linearly independent in the domain under consideration, and $f_{ij}=[f_i,f_j]$ with $i,j\in\{0,1,\hdots,k\}$.
\begin{notation}
\label{246} Recalling Notation \ref{notaz}, let us introduce the following abbreviated notation: $H_{0I}:=H_{0I}(\bar{\lambda}),\
H_{IJ}:=H_{IJ}(\bar{\lambda})$, chosen an opportune $\bar{\lambda}\in\Lambda|_{\bar q}$.
\end{notation}
In order to prove Theorem \ref{codim1thm}, we are going to study extremals for any control system of the form (\ref{111}) with $ k<n$ in a neighbourhood of $\bar{\lambda}\in\Lambda_{\bar{q}}\subseteq T^*_{\bar{q}}M$  such that
\begin{equation}
\label{cond_eqq}
H_{0I}\notin H_{IJ}\,S^{k-1},
\end{equation}
where $S^{k-1}=\{u\in \mathbb{R}^k\,:\, ||u||=1\}$ is the unit sphere.
\begin{remark}
If $k=n-1$, we should choose $\bar{\lambda}=f_1(\bar{q})\wedge\hdots\wedge f_{n-1}(\bar{q})$. One can notice that conditions (\ref{codim1}) and (\ref{cond_eqq}) are equivalent.
\end{remark}
From Corollary \ref{017} we already know that every arc of a time-optimal trajectory, whose extremal lies out of $\Lambda$, is bang, and so smooth. \\
Thus, we are interested to study arcs of a time-optimal trajectories, whose extremals passes through $\Lambda$ or lies in $\Lambda$.\\
The fist step is to investigate if our system admits singular arcs.
\begin{proposition}
\label{809}
Assuming (\ref{cond_eqq}), there are no optimal extremals in $O_{\bar{\lambda}}$ that lie in the singular locus $\Lambda$ for a time interval.
\end{proposition}
Thanks to Proposition \ref{809}, if it holds (\ref{cond_eqq}), the description of optimal extremals in a neighbourhood of $\bar{\lambda}$ is essentially reduced to the study of the solutions of the Hamiltonian system with a discontinuous right-hand side, defined by the Hamiltonian $\mathcal{H}(\lambda)=h_{0}(\lambda)+\sqrt{h^2_{1}(\lambda)+\hdots+h^2_{k}(\lambda)}$.

\begin{theorem}
\label{resultneq}
Assume that condition (\ref{cond_eqq}) is satisfied.\\
If it holds
\begin{equation}
\label{condnotin}
H_{0I}\notin H_{IJ}\overline{B^{k}},
\end{equation}
where $B^k=\{u\in \mathbb{R}^k\,:\,||u||<1\}$, then there exists a neighborhood $O_{\bar\lambda}\subset T^*M$ such that for any $z\in O_{\bar\lambda}$ and $\hat{t}>0$ there exists a unique contained in $O_{\bar\lambda}$ extremal $t\mapsto\lambda(t,z)$ with the condition $\lambda(\hat{t},z)=z$. Moreover, $\lambda(t,z)$ continuously depends on $(t,z)$ and every extremal in $O_{\bar{\lambda}}$ that passes through the singular locus is piece-wise smooth with only one switching.\\
Besides that, if $u$ is the control corresponding to the extremal that passes through $\bar{\lambda}$, and $\bar{t}$ is its switching time, we have:
\begin{equation}
\label{jumpi}
u(\bar{t}\pm 0)=[\pm d\,\mathrm{Id}+H_{IJ}]^{-1}H_{0I},
\end{equation}
with $d\,>0$ unique, uni vocally defined by the system and $\bar{\lambda}$, such that
\begin{equation}
\label{iii}\left\langle [d^2\, \mathrm{Id}-H_{IJ}^2]^{-1}H_{0I},H_{0I} \right\rangle=1 .
\end{equation}
If it holds
\begin{equation}
\label{condin}
H_{0I}\in H_{IJ}B^{k},
\end{equation}
then there exists a neighbourhood $O_{\bar\lambda}\subset T^*M$ such that no one optimal extremal intersects singular locus in $O_{\bar\lambda}$.
\end{theorem}
Note that $H_{IJ}\overline{B^{k}}=H_{IJ}S^{k-1}$ if the matrix $H_{IJ}$ is degenerate, and that this matrix is always degenerate for odd $k$. Hence, assuming (\ref{cond_eqq}), we have the following possibilities:\\
\begin{center}
\begin{tabular}{l}
\framebox[10cm]{\begin{minipage}{90mm}
It holds (\ref{condnotin}) if it is verified one of the following scenarios:
\begin{itemize}
\item[$(A)$]$k$ is odd
\item[$(B)$]$k$ is even and $H_{IJ}$ is degenerate
\item[$(C')$]$k$ is even, $H_{IJ}$ is non-degenerate and $H_{0I}\notin H_{IJ}\overline{B^{k}}$ .
\end{itemize}
\end{minipage}} \\ \\
\framebox[10cm]{\begin{minipage}{9cm}
It holds (\ref{condin}) if it is verified the following scenario:
\begin{itemize}
\item[$(C'')$]$k$ is even, $H_{IJ}$ is non-degenerate and $H_{0I}\in H_{IJ}B^{k}$ .
\end{itemize}
\end{minipage}}
\end{tabular}
\end{center}

\begin{remark}
 In general, the flow of switching extremals from Theorem \ref{resultneq} is not locally Lipschitz with respect to the initial value. In \cite{AB} was found a simple counterexample that can be easily generalized to any $k<n$.
\end{remark}
Since the Pontryagin Maximum Principle is a necessary but not sufficient condition of optimality, even if we have found extremals that passes through the singular locus, we cannot guaranty that they are all optimal, namely that their projections in $M$ are time-optimal trajectory. In some cases they are certainly optimal, in particular, for linear system with an equilibrium target, where to be an extremal is sufficient for optimality. We plan to study general case in a forthcoming paper.\\

\section{Proof}
\label{Proof}
In this Section we are going to present at first the proof of Theorem \ref{resultneq}, secondly we are going to prove Proposition \ref{809}. All together, these statements contain Theorem \ref{codim1thm}.
\subsection{Proof of Theorem \ref{resultneq}} Let us present the Blow-up technique, in order to analyse the discontinuous right-hand side Hamiltonian system, defined by
\begin{equation}
\label{111098}
\mathcal{H}(\lambda)=h_0(\lambda) +  \sqrt{h_1^2(\lambda)+\dots+h_{k}^2(\lambda)},
\end{equation}in a neighbourhood $O_{\bar{\lambda}}$ of $\bar{\lambda}$.
\subsubsection{Blow-up technique}In view of the fact that this is a local problem in $O_{\bar{\lambda}}\subseteq T^*M$, it is very natural consider directly its local coordinates $(\xi,x)\in \mathbb{R}^{n*}\times \mathbb{R}^n$, such that $\bar{\lambda}$ corresponds to $(\bar{\xi}, \bar{x})$ with $\bar{x}=0$. Hence,
\begin{equation}
\label{1110988}
\mathcal{H}(\xi,x)=h_0(\xi,x) +  \sqrt{h_1^2(\xi,x)+\hdots+h_k^2(\xi,x)}.
\end{equation}
Since $f_1, \dots, f_k$ are linearly independent everywhere, we can define $n-k$ never null vector fields $f_{k+1},\dots,f_n$, such that $\{f_1,\hdots,f_n\}$ form a basis at any $q\in M$, then we will have the corresponding $h_j(\xi,x)=\left\langle \xi,f_j(x) \right\rangle $, with $j=k+1,\hdots, n$. Therefore, we are allowed to consider the following smooth change of variables
$$
\Phi:\, (\xi,x)\longrightarrow  ((h_1,\hdots,h_{n}),x),
$$
so the singular locus becomes the subspace $$\Lambda=\{((h_1,\hdots,h_{n}),x) \, :\, h_1=\hdots=h_k=0 \}.$$
\begin{notation}In order not to do notations even more complicated, we call $\lambda$ any point defined with respect to the new coordinates $((h_1,\hdots,h_{n}),x)$, and $\bar{\lambda}$ what corresponds to the singular point.
\end{notation}
Thus, let us define the blow-up technique.
\begin{definition}
The \emph{blow-up} technique is defined in the following way:\\
We make a change of variables: $(h_1,\hdots, h_k)=(\rho u_1,\hdots,\rho u_k)$ with $\rho\in\mathbb{R}^+$ and $(u_1,\hdots,u_k)\in S^{k-1}$. Instead of considering the components $h_1,\hdots,h_k$ of the singular point $\bar{\lambda}$ in $\Lambda$, as the point $(0,\hdots,0)$ in the k-dimensional euclidean space, we will consider it as a sphere $S^{k-1}$, where $\{\rho=0\}$.
\end{definition}
\begin{center}
\begin{pspicture}(0,-1.4812988)(13.006666,1.4812988)
\definecolor{colour0}{rgb}{0.8,0.8,0.8}
\rput[bl](3.36,-1.2653614){$\textcolor{black}{\bar{\lambda}}$}
\rput[bl](5.8133335,0.09463871){\large}
\rput[bl](5.8133335,0.09463871){Blow - up}
\rput[bl](7.6,-1.4053613){$\textcolor{black}{\bar{\lambda}}$}
\rput[bl](7.846667,0.57463866){$\textcolor{black}{\bar{\lambda}_{u}}$}
\psdots[linecolor=black, dotsize=0.18](3.7666667,-0.25869465)
\psline[linecolor=black, linewidth=0.04, linestyle=dashed, dash=0.17638889cm 0.10583334cm, arrowsize=0.05291667cm 2.0,arrowlength=1.4,arrowinset=0.0]{->}(5.32,-0.25869465)(7.68,-0.25869465)
\psdots[linecolor=black, dotsize=0.18](8.78,-0.005361328)
\psline[linecolor=black, linewidth=0.04, linestyle=dotted, dotsep=0.10583334cm, arrowsize=0.05291667cm 2.0,arrowlength=1.4,arrowinset=0.0]{->}(3.5133333,-0.8453613)(3.6866667,-0.43202797)
\psline[linecolor=black, linewidth=0.04, linestyle=dotted, dotsep=0.10583334cm, arrowsize=0.05291667cm 2.0,arrowlength=1.4,arrowinset=0.0]{->}(8.14,-1.0453613)(8.46,-0.6586947)
\psline[linecolor=black, linewidth=0.04, linestyle=dotted, dotsep=0.10583334cm, arrowsize=0.05291667cm 2.0,arrowlength=1.4,arrowinset=0.0]{->}(8.06,0.5546387)(8.566667,0.18130533)
\rput[bl](0.0,1.2146386){\Large}
\rput[bl](0.0,1.2146386){$\textcolor{black}{((h_1,\hdots,h_k,h_{k+1},\hdots,h_n), x)}$}
\psline[linecolor=black, linewidth=0.037, arrowsize=0.05291667cm 2.0,arrowlength=1.4,arrowinset=0.0]{->}(5.32,1.4146386)(7.68,1.4146386)
\rput[bl](7.846667,1.2146386){\Large}
\rput[bl](7.846667,1.2146386){$\textcolor{black}{((\rho,u_1,\hdots,u_k, h_{k+1},\hdots, h_n),x)}$}
\pscircle[linecolor=black, linewidth=0.04, fillstyle=solid,fillcolor=colour0, dimen=outer](8.85,-0.11536138){0.41}
\psdots[linecolor=black, dotsize=0.16](8.78,0.05463862)
\end{pspicture}

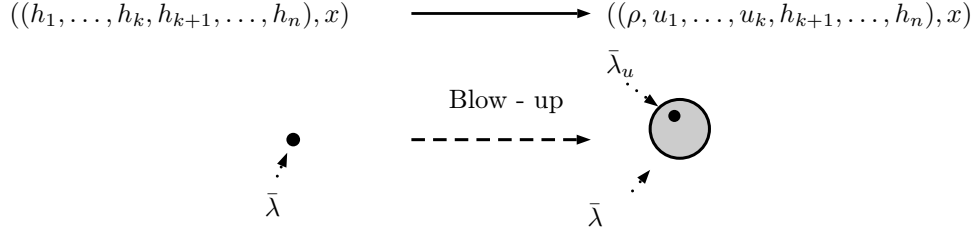
\captionof{figure}{Blow-up technique}
\end{center}
\bigskip
Let us notice that it is good to denote $u:=(u_1,\hdots,u_k)$ the $S^{k-1}$-coordinates. As it is already know from Proposition \ref{776}, every optimal control $\tilde{u}$, that  corresponds to an extremal $\lambda(t)$ out of $\Lambda$, satisfies formula (\ref{time-opt.control}): therefore $\tilde{u}$ lies on $\partial U=S^{k-1}$, and it is the normalization of the vector $(h_1(\lambda(t)),\hdots,h_{k}(\lambda(t)))$.\\
It is useful denote
$$ f_{u}(x)=u_1 f_1(x)+\hdots+u_k f_k(x) $$
and $h_{u}(\lambda)=\left\langle  \xi, f_{u}(x) \right\rangle $; and finally we can see that $$h_{u}(\lambda)=\sqrt{h^2_1+\hdots+h^2_k},$$ namely $h_{u}(\lambda)=\rho$, because $h_{u}(\lambda)=u_1h_1+\hdots+u_k h_k$, and $u_i=\frac{h_i}{\sqrt{h^2_1+\hdots+h^2_k}}$ for all $i\in\{1,\hdots,k\}$.\\
Hence, with this new formulation the maximized Hamiltonian becomes
\begin{equation}
\label{111099888}
\mathcal{H}(\lambda)=h_0(\lambda)+  h_{u}(\lambda).
\end{equation}
Thanks to Notation \ref{notaz}, Remarks \ref{derivPoiss} and \ref{poisson,lie}, the Hamiltonian system has the following form:
\begin{equation}
\label{11100}
\left\lbrace \begin{array}{l}
\dot{x}=f_0(x)+f_{u}(x)\\
\dot{\rho} = \left\langle H_{0I}(\lambda),u \right\rangle  \\
\dot{u} = \frac{1}{\rho}\left( H_{0I}(\lambda)-\left\langle H_{0I}(\lambda),u \right\rangle u-H_{IJ}(\lambda)u \right)\\
\dot{h}_{j}=h_{0j}(\lambda)+h_{u j}(\lambda), \quad j\in\{k+1,\hdots,n\}.
\end{array}\right.
\end{equation}
\begin{claim}
\label{349}
If assumption (\ref{condnotin}) is satisfied at the singular point $\bar{\lambda}$, then in $S^{k-1}$
\begin{equation}
\label{gg}
u\longmapsto H_{0I}-\left\langle H_{0I},u \right\rangle u-H_{IJ}\,u
\end{equation}
has two zeros $u_+$ and $u_-$ defined in the following way:
\begin{equation}
\label{coc}
u_\pm=[\pm d\, \mathrm{Id}+H_{IJ}]^{-1}H_{0I},
\end{equation}
with $d>0$ such that
\begin{equation}
\label{cocc}
\left\langle [d^2\, \mathrm{Id}-H_{IJ}^2]^{-1}H_{0I},H_{0I} \right\rangle=1 .
\end{equation}
The function (\ref{gg}) has no zero if it holds assumption (\ref{condin}).
\end{claim}
\begin{proof}
Denoting $Z:=\left\langle  H_{0I},u\right\rangle $, we are looking for $u\in S^{k-1}$ and $Z\in\mathbb{R}$ such that
$$H_{0I}=(Z \,\mathrm{Id}+H_{IJ})u.
$$
We already know that, if $Z=0$, then there is no $u\in S^{k-1}$ such that $H_{0I}=H_{IJ}u$, by assumption (\ref{cond_eqq}). Moreover, since $H_{IJ}$ is a skew-symmetric matrix, if $Z\neq 0$ then $(Z \,\mathrm{Id} +H_{IJ})$ is invertible, and
$$u=(Z\,\mathrm{Id}+H_{IJ})^{-1}H_{0I}.
$$
Let us consider the function
\begin{equation}
\label{newfunct}Z\longmapsto || (Z\,\mathrm{Id}+H_{IJ})^{-1}H_{0I} ||^2
\end{equation}
that will be continuous even and monotone in the domains $(-\infty,0)$ and $(0,+\infty)$, because $$|| (Z\,\mathrm{Id}+H_{IJ})^{-1}H_{0I} ||^2=\left\langle [Z^2\, \mathrm{Id}-H_{IJ}^2]^{-1}H_{0I},H_{0I} \right\rangle,$$ 
and its derivation with respect to $Z^2$ is negative
$$\frac{\mathrm{d}}{\mathrm{d}(Z^2)}\left\langle [Z^2\, \mathrm{Id}-H_{IJ}^2]^{-1}H_{0I},H_{0I} \right\rangle<0.$$ 
Indeed, it holds
$$\begin{array}{rcl}
\frac{\mathrm{d}}{\mathrm{d}(Z^2)}\left\langle [Z^2\, \mathrm{Id}-H_{IJ}^2]^{-1}H_{0I},H_{0I} \right\rangle&=&-\left\langle [Z^2\, \mathrm{Id}-H_{IJ}^2]^{-2}H_{0I},H_{0I} \right\rangle\\
&&\\
&=&-|| [Z^2 \mathrm{Id}- H_{IJ}^2]^{-1}H_{0I} ||^2.
\end{array}
$$
We are going to verify if and in which cases the function (\ref{newfunct}) takes value 1 two or zero times. Thus, let us compute the limits of $|| (Z\,\mathrm{Id}+H_{IJ})^{-1}H_{0I} ||^2$ as $Z\rightarrow \pm \infty$ or $Z\rightarrow 0^\pm$.

At first one can observe that, 
$$\lim_{Z\rightarrow \pm \infty}|| (Z\,\mathrm{Id}+H_{IJ})^{-1}H_{0I} ||^2= 0^+.$$

In order to compute $\lim_{Z\rightarrow 0^\pm}|| (Z\,\mathrm{Id}+H_{IJ})^{-1}H_{0I} ||^2$, let us assume that $H_{IJ}$ is in the canonical Jordan form, without loss of generality: it is defined by $j\leq \lfloor\frac{n}{2}\rfloor$ $2\times 2$ skew symmetric  blocks with the following form
$$J_i=\left(\begin{array}{cc}
0&a_i\\
-a_i&0
\end{array}\right)\quad i\in\{1,\hdots,j\},
$$ and the rest of the matrix is null.\\

Let $H_{IJ}$ be a degenerate matrix. If $H_{0I}$ does not belong to its image, namely $H_{0I}\notin H_{IJ}\mathbb{R}^k$, it holds
$$\lim_{Z\rightarrow 0^\pm}|| (Z\,\mathrm{Id}+H_{IJ})^{-1}H_{0I} ||^2= +\infty.$$ 
On the other hand, let us show that if $H_{0I}\in H_{IJ}\mathbb{R}^k$ the  limit $\lim_{Z\rightarrow 0^\pm}|| (Z\,\mathrm{Id}+H_{IJ})^{-1}H_{0I} ||^2$ is finite strictly grater that 1.\\
Since $H_{IJ}$ is degenerate, it holds $H_{IJ}\overline{B}^k=H_{IJ}S^{k-1}$, then by condition (\ref{cond_eqq}) we have $H_{0I}\notin H_{IJ}\overline{B}^k$. Thus, given condition $H_{0I}\in H_{IJ}\mathbb{R}^k$ we have that for all $X$, such that $H_{0I}= H_{IJ}\,X$, it has norm strictly grater than $1$.\\
Finally, let us define 
$$X=\left(
\begin{array}{cccccc}
J_1^{-1}&&\\
&\ddots&&\\
&&J_j^{-1}&\\
&&&0_{(n-2j)\times(n-2j)}
\end{array}
\right)H_{0I},
$$
and see, by construction, that 
$$\lim_{Z\rightarrow 0^\pm}|| (Z\,\mathrm{Id}+H_{IJ})^{-1}H_{0I} ||^2= ||X||^2>1.$$

Hence, if $H_{IJ}$ is degenerate, by monotonicity and continuity of (\ref{newfunct}),  there will be a value $Z=d>0$ such that $$|| (\pm d\,\mathrm{Id}+H_{IJ})^{-1}H_{0I} ||^2=1.$$ It means that there exist $u_+$ and $u_-$ zeros of the function (\ref{gg}) such that $|\left\langle  H_{0I},u_\pm\right\rangle |=d$. We will assume $\left\langle  H_{0I},u_+\right\rangle >0$ and $\left\langle  H_{0I},u_-\right\rangle <0$.\\
These facts happen in scenarios (A) and (B) of condition (\ref{condnotin}).\\

If $H_{IJ}$ is a non-degenerate matrix, then (\ref{newfunct}) is a continuous function for all $Z\in \mathbb{R}$ and
$$\lim_{Z\rightarrow 0}|| (Z\,\mathrm{Id}+H_{IJ})^{-1}H_{0I} ||^2= || H_{IJ}^{-1}H_{0I} ||^2.$$
Thus, in this case the function (\ref{gg}) will have two or no zeros if and only if $|| H_{IJ}^{-1}H_{0I} ||>1$ or $|| H_{IJ}^{-1}H_{0I} ||<1$, namely $H_{0I}\notin H_{IJ}\overline{B^k}$ or $H_{0I}\in H_{IJ}B^k$. \\
These are, indeed, scenarios $(C')$ and $(C'')$.
\end{proof}
\bigskip
\subsubsection{Case $H_{0I}\in H_{IJ}B^k$}\textcolor{white}{.}\\
Once we have seen that (\ref{gg}) have no zero in this case, let us present the following Lemma in order to prove Theorem \ref{resultneq} if $H_{0I}\in H_{IJ}B^k$.
\begin{lemma}
\label{lemma1}
Let us assume  (\ref{cond_eqq}), (\ref{condin}) and give a neighbourhood $O_{\bar{\lambda}}$ small enough such that
$$H_{0I}(\lambda)-\left\langle H_{0I}(\lambda),u \right\rangle u-H_{IJ}(\lambda)u\neq 0,\quad \forall \lambda_u\in \overline{O_{\bar{\lambda}}}.$$
Then there exist two constants $c>0$ and $\alpha>0$ such that every optimal extremal that lies for a time interval $I\subseteq [0,+\infty)$ in $O_{\bar{\lambda}}$ satisfies the following inequality: $\rho(t)\geq c e^{-\alpha (t)}\rho(0)$, for $t\in I$.
\end{lemma}
\begin{proof}
Let us call
\begin{equation}
\label{509}
v(\lambda)=H_{0I}(\lambda)-\left\langle H_{0I}(\lambda),u \right\rangle u-H_{IJ}(\lambda)u,
\end{equation}
by construction, we can assume that for all $\lambda\in \overline{O_{\bar{\lambda}}}$ it holds
$$||v(\lambda)|| > 0.$$
Since in the compact set $\overline{O_{\bar{\lambda}}}$ the map $\lambda \rightarrow v(\lambda)$ is continuous and not null, then there exist constants $c_1>0$ and $c_2>0$ such that, for all $\lambda\in O_{\bar{\lambda}}$,
$$c_1 \geq ||v(\lambda)||\geq c_2>0.
$$
Given the extremal $\lambda(t)$ in $O_{\bar{\lambda}}$, we can observe that
$$\frac{d}{dt}\rho(t)||v(\lambda(t))|| =\rho(t)\frac{\left\langle v(\lambda(t)),A(\lambda(t))\right\rangle }{||v(\lambda(t))||}=\rho(t)\tilde{A}(\lambda(t))
$$
where
$$A(\lambda(t))=\dot{H}_{0I}(\lambda(t))-\left\langle \dot{H}_{0I}(\lambda(t)),u(t) \right\rangle u(t)-\dot{H}_{IJ}(\lambda(t))u(t).$$
Let us notice that for any Hamiltonian $h(\lambda)$ its time-derivative along $\lambda(t)$ is
$$\begin{array}{rcl}
\dot{h}(\lambda(t))=\{h_0+\rho,h\}(\lambda(t))&=&\{h_0,h\}(\lambda(t))+\{\rho,h\}(\lambda(t))\\
&=&\{h_0,h\}(\lambda(t))+\frac{1}{\rho}\sum^k_{i=1}h_i(\lambda(t))\{h_i,h\}(\lambda(t))\\
&=&\{h_0,h\}(\lambda(t))+\sum^k_{i=1}u_i(t)\{h_i,h\}(\lambda(t))\\
\end{array}
$$
and it is bounded.\\
As a consequence each component of $A(\lambda(t))$ is bounded too, and $\tilde{A}_{|O}$ is bounded from below by a negative constant $C$
$$\tilde{A}_{|O}\geq C.
$$
Finally, we can see that
$$\frac{d}{dt}\left[\frac{\rho(t)||v(\lambda(t))||}{\mathrm{exp}\left( \int^t_0 C \left[||v(\lambda(s))||\right]^{-1}ds \right)}\right]\geq 0,
$$
hence, for each $t\geq \tau_1$, by the monotonicity:
$$
\begin{array}{rcl}
\rho(t)&\geq& \rho(\tau_1)\,\,\frac{||v(\lambda(\tau_1))||}{||v(\lambda(t))||}\,\,\mathrm{exp}\left( \int^t_{\tau_{1}} C \left[||v(\lambda(s))||\right]^{-1}ds \right)\\
&\geq&\rho(\tau_1)\,\,\frac{c_2}{c_1}\,\,\mathrm{exp}\left( \frac{C}{c_2}(t-\tau_1) \right).
\end{array}
$$
Denoting $c:=\frac{c_2}{c_1}$ and $\alpha:=-\frac{C}{c_2}$, the thesis follows.
\end{proof}
This Lemma proves Theorem \ref{resultneq} if $H_{0I}\in H_{IJ}B^k$, because it shows that, given those conditions, every optimal extremal in $O_{\bar{\lambda}}$ does not intersect the singular locus  in finite time, and forms a smooth local flow.
\subsubsection{Case $H_{0I}\notin H_{IJ}\overline{B^k}$}
\begin{proposition}
Given condition (\ref{cond_eqq}) and assumption (\ref{condnotin}), there exists a unique extremal that passes through $\bar{\lambda}$ in finite time.
\end{proposition}
\begin{proof}
Let us prove that there is a unique solution of the system (\ref{11100}) passing through its point of discontinuity $\bar{\lambda}$ in finite time.\\
In order to detect solutions that go through $\bar{\lambda}$, we rescale the time considering the time $t(s)$ such that $\frac{d}{ds}t(s)=\rho(s)$ and we obtain the following system
\begin{equation}
\label{11100s}
\left\lbrace \begin{array}{l}
x'=\rho\left(f_0(x)+f_{u}(x)\right)\\
\rho' = \rho\left\langle H_{0I}(\lambda),u \right\rangle  \\
u' =  H_{0I}(\lambda)-\left\langle H_{0I}(\lambda),u \right\rangle u-H_{IJ}(\lambda)u \\
{h'}_{j}=\rho\left(h_{0j}(\lambda)+h_{u j}(\lambda)\right), \quad j\in\{k+1,\hdots,n\}.
\end{array}\right.
\end{equation}
with a smooth right-hand side.\\
This system has an invariant subset $\{\rho=0\}$ in which only the $u$-component is moving. Moreover, as we saw from Claim \ref{349}, at $\bar{\lambda}\in \{\rho=0\}$ there are two equilibria $\bar{\lambda}_{u_-}$ and $\bar{\lambda}_{u_+}$, such that $\left\langle H_{0I},u_+ \right\rangle  >0$ and $\left\langle H_{0I},u_- \right\rangle  <0$.\\

Let us present the Shoshitaishvili's Theorem \cite{sh} that explain how is the behaviour of the solutions in $O_{\bar{\lambda}_{u_-}}$ and $O_{\bar{\lambda}_{u_+}}$ neighbourhoods of the equilibria $\bar{\lambda}_{u_-}$ and $\bar{\lambda}_{u_+}$ in $T^*M$.
\begin{theorem}[Shoshitaishvili's Theorem]
\label{shosh}
In a n-dimensional manifold $N$ with $\lambda\in N$, let 
\begin{equation}
\label{202}
\dot{\lambda}=f(\lambda)
\end{equation}
a dynamical system in $N$, where $f\in\mathcal{C}^k(N)$, $2\leq k<\infty$. Given $\bar{\lambda}\in N$ there exists an opportune neighbourhood $O_{\bar{\lambda}}$ such that, via the coordinate chart, (\ref{202}) is described by the following system in $\mathbb{R}^n$
\begin{equation}
\label{222}
\dot{z}=Bz+r(z),\quad  z\in\mathbb{R}^n,
\end{equation}
where $r\in C^k(\mathbb{R}^n)$, $r(0)=0$, $\partial_{z}r_{|0}=0$, and $B:\mathbb{R}^n\rightarrow \mathbb{R}^n$ is a linear operator whose eigenvalues are divided into three groups:
$$
\begin{array}{c}
\begin{array}{rcl}
\mathrm{I}&=&\{ \mu_i, 1\leq i \leq k^0|\, \mathrm{Re}\mu_i=0 \}\\
\mathrm{II}&=&\{ \mu_i, k^0+1\leq i \leq k^0+k^-|\,\mathrm{Re}\mu_i<0 \}\\
\mathrm{III}&=&\{ \mu_i, k^0+k^-+1\leq i \leq k^0+k^-+k^+|\,\mathrm{Re}\mu_i>0 \}
\end{array}
\\
\\
\\
k^0+k^-+k^+=n.
\end{array}
$$
Let the subspaces of $\mathbb{R}^n$, which are invariant with respect $B$ and which correspond to these groups be denoted by $X$, $Y^-$ and $Y^+$ respectively, and let $Y^-\times Y^+$ be denoted by $Y$.\\
Then the following assertions are true:\\
\begin{enumerate}
\item There exists a $\mathcal{C}^{k-1}$ manifold $\gamma^0$ that is invariant with respect to (\ref{202}), may be given by the graph of mapping $\gamma^0:X\rightarrow Y$, $y=\gamma^0(x)$, and satisfies $\gamma^0(0)=0$ and $\partial_x \gamma^0(0)=0$.
\item The system (\ref{202}) in $O_{\bar{\lambda}}$ is homeomorphic to the product of the multidimensional saddle $\dot{y}^+=y^+$, $\dot{y}^-=-y^-$, and 
$$
\dot{x}=\hat{B}x+r_1(x)
$$
where $r_1(x)$ is the $x$-component of the vector $r(z)$, $z=(x,\gamma^0(x))$, i.e. (\ref{202})in $O_{\bar{\lambda}}$ is homeomorphic to the system
$$\left\lbrace \begin{array}{l}
\begin{array}{ll}
\dot{y}^+=y^+, & \dot{y}^-=-y^-
\end{array}\\
\begin{array}{ll}
\dot{x}=\hat{B}x+r_1(x).
\end{array}
\end{array} \right.
$$
\end{enumerate}
\end{theorem}
Due to the fact that $\bar{\lambda}_{u_-}$ and $\bar{\lambda}_{u_+}$ belong to the invariant subset $\{
\rho=0\}$, where the components $\rho$, $h_j$ with $j\in\{k+1,\hdots,n\}$ and $x$ are fixed, we can observe that Jacobian matrix of (\ref{11100s}) have the following eigenvalues: $\left\langle H_{0I},u_{\pm} \right\rangle $ that corresponds to the $\rho$-coordinate, the eigenvalues of the matrix ${\partial_{u}v(\bar{\lambda}_u)}_{|\bar{\lambda}_{u_\pm}}$, recalling notation (\ref{509}), that correspond to the $u$-coordinate, and $2n-k$ $0$-eigenvalues corresponding to the other coordinates.\\

Thus, let us study ${\partial_{u}v(\bar{\lambda}_u)}_{|\bar{\lambda}_{u_\pm}}$ that has the following form
\begin{equation}
\label{ooo}{\partial_{u}v(\bar{\lambda}_u)}_{|\bar{\lambda}_{u_\pm}}=-\left[ \left\langle H_{0I},u_{\pm} \right\rangle \mathrm{Id}+H_{IJ}+u_\pm\,H_{0I}^T \right]
\end{equation}
where $H_{0I}^T$ is the row vector. \\
Let us prove that the real part of its eigenvalues is equal $-\left\langle  H_{0I},u_\pm\right\rangle $.\\

Let $\alpha+i \beta$ be an eigenvalue of ${\partial_{u}v(\bar{\lambda}_u)}_{|\bar{\lambda}_{u_\pm}}$ with $w_R+iw_I\neq 0$ eigenvector, as a consequence we can claim that
$$\left\lbrace
\begin{array}{l}
{\partial_{u}v(\bar{\lambda}_u)}_{|\bar{\lambda}_{u_\pm}}w_R=\alpha w_R-\beta w_I\\
{\partial_{u}v(\bar{\lambda}_u)}_{|\bar{\lambda}_{u_\pm}}w_I=\alpha w_I +\beta w_R.
\end{array}
 \right. 
$$
Thus, it holds $\left\langle {\partial_{u}v(\bar{\lambda}_u)}_{|\bar{\lambda}_{u_\pm}}w_R,w_R\right\rangle+\left\langle {\partial_{u}v(\bar{\lambda}_u)}_{|\bar{\lambda}_{u_\pm}}w_I,w_I\right\rangle=\alpha (|w_R|^2+|w_I|^2)$, and it implies $$-\left\langle  H_{0I},u_\pm\right\rangle (|w_R|^2+|w_I|^2)=\alpha(|w_R|^2+|w_I|^2),$$
because $w_R$ and $w_I$ are orthogonal to $u_\pm$. Since $w_R+iw_I\neq 0$, it holds $$\alpha =-\left\langle  H_{0I},u_\pm\right\rangle.$$

By Claim \ref{349}, we know that $\left\langle H_{0I},u_{-} \right\rangle$ and $\left\langle H_{0I},u_{+} \right\rangle$ are not null with opposite sign. Hence, assuming $\left\langle H_{0I},u_{-} \right\rangle<0$, we can conclude that in a neighbourhood of $\bar{\lambda}_{u_-}$ there is a stable 1-dimensional submanifold with respect to $\rho$ and an unstable submanifold with respect to $u$. Analogously in a neighbourhood of $\bar{\lambda}_{u_+}$, we can notice the unstable 1-dimensional submanifold with respect to $\rho$ and the stable one with respect to $u$.

Central manifolds $\gamma^0$ of Theorem \ref{shosh} applied to the equilibria $\bar\lambda_{u_\pm}$ are ($2n-k$)-dimensional submanifolds defined by the equations $\rho=0,\ u=u_\pm$. The dynamics on the central manifold is trivial: all points are equilibria.

Hence, according to the Shoshitaishvili Theorem, there is a trajectory from the one-dimensional asymptotically stable invariant submanifold that tends to the equilibrium point
$\bar\lambda_{u_-}$ as $s\to +\infty$, and analogously there is a trajectory from the one-dimensional asymptotically unstable invariant submanifold that escapes from the equilibrium point
$\bar\lambda_{u_+}$ as $s\to -\infty$.\\

In order to obtain that exactly one solution of (\ref{11100s}) enters submanifold $\rho=0$ at $\bar\lambda_{u_-}$ and exactly one goes out of this submanifold at $\bar\lambda_{u_+}$, let us present together with Shoshitaishvili Theorem the following Proposition \ref{7771}, that  shows the behaviour of solutions  with rescaled time $s$, in the subset $\{\rho=0\}$ where only the $u$-component is moving with respect to the equation 
\begin{equation}
\label{777}
u' =  H_{0I}-\left\langle H_{0I},u \right\rangle u-H_{IJ}u.
\end{equation}
Then it is completely described the whole phase portrait of the system (\ref{11100s}).\\
\begin{center}
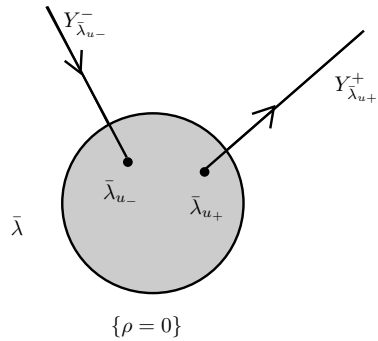

\psscalebox{.8 .8} 
{
\begin{pspicture}(0,-2.7826421)(5.74,2.7826421)
\definecolor{colour0}{rgb}{0.8,0.8,0.8}
\pscircle[linecolor=black, linewidth=0.04, fillstyle=solid,fillcolor=colour0, dimen=outer](2.37,-0.4867046){1.51}
\psline[linecolor=black, linewidth=0.04](0.62,2.7732954)(1.94,0.2732954)
\psline[linecolor=black, linewidth=0.04](5.84,2.3732953)(3.18,0.033295404)
\psdots[linecolor=black, dotsize=0.16](3.22,0.033295404)
\psdots[linecolor=black, dotsize=0.16](1.96,0.1932954)
\psline[linecolor=black, linewidth=0.04](0.92,1.8732954)(1.18,1.6732954)(1.2,2.0532954)
\psline[linecolor=black, linewidth=0.04](4.0,0.99329543)(4.36,1.1132954)(4.3,0.7532954)
\psline[linecolor=black, linewidth=0.03](3.58,0.39329544)(5.8,2.3332953)
\rput[bl](2.94,-0.28670457){\Large}
\rput[bl](2.94,-0.78670457){$\bar{\lambda}_{u_+}$}
\rput[bl](1.52,-0.033569816){\Large}
\rput[bl](1.52,-0.533569816){$\bar{\lambda}_{u_-}$}
\psline[linecolor=black, linewidth=0.03](1.64,0.8132954)(0.64,2.7732954)
\rput[bl](5.32,1.1732954){$Y^+_{\bar{\lambda}_{u+}}$}
\rput[bl](0.86,2.2332954){$Y^-_{\bar{\lambda}_{u-}}$}
\rput[bl](1.64,-2.7067046){\Large}
\rput[bl](1.64,-2.7067046){$\{\rho=0\}$}
\rput[bl](0.0,-1.0067046){\Large}
\rput[bl](0.0,-1.0067046){$\bar{\lambda}$}
\end{pspicture}
}

\captionof{figure}{Solution of (\ref{11100s}) that passes through $\bar{\lambda}\in \Lambda$. }
\end{center}
\bigskip
\begin{proposition}
\label{7771} Let $u(s),\ s\in\mathbb R$, be a solution of system (\ref{777}) that is not an equilibrium. Then
$u(s)\to u_{\pm}$ as $s\to\pm\infty$.
\end{proposition}
\begin{proof}
 Let $y(t)$ be a solution of the system $\dot y=|y|H_{0I}-H_{IJ}y,\ y\in\mathbb R^k$, then
$u(t)=\frac 1{|y(t)|}y(t)$ satisfies system (\ref{777}). Consider a linear $(k+1)$-dimensional system
\begin{equation}
\label{2co}\dot x=\langle H_{0I},y\rangle, \quad \dot y=xH_{0I}-H_{IJ}y.
\end{equation}
Its solutions preserve the Lorentz form $Q(x,y)=x^2-|y|^2$ and, in particular, the cone
$$
C=\{(x,y)\in\mathbb R^{k+1}:x^2=|y|^2\}.
$$
We obtain that $s\mapsto y(s)$ is a solution of system $\dot y=|y|H_{0I}-H_{IJ}y,\ y\in\mathbb R^k$ if and only if 
$s\mapsto (|y(s)|,y(s))$ is a solution of (\ref{2co}).

System (\ref{2co}) has a form $\dot z=Bz$, where $z=(x,y)$ and $B$ is a \linebreak $(k+1)\times(k+1)$-matrix.
Moreover, vectors $(1,u_\pm)$ are eigenvectors of the matrix $B$ with eigenvalues $\langle H_{0I},u_\pm\rangle$.
System $\dot z=Bz$ preserves any invariant subspace of $B$ and in particular hyperplanes
$T_{(1,u_\pm)}C$. Note that the projectivization of $C$ is a strictly convex cone, hence
$C\cap T_{(1,u_\pm)}C=\mathrm{span}\{(1,u_\pm)\}$. 

We obtain that a co-dimension two subspace $E=T_{(1,u_\pm)}C\cap T_{(1,u_\pm)}C$ has zero
intersection with $C$. It follows that quadratic form $Q$ is sign-definite on the subspace $E$. Hence all solutions of system $\dot z=Bz$ that belong to the invariant subspace $E$ are bounded for both positive and negative time. Any solution of system $\dot z=Bz$ has a form:
$$
s\mapsto c_+e^{s\langle H_{0I},u_+\rangle}(1,u_+)+c_-e^{s\langle H_{0I},u_-\rangle}(1,u_-)+e(s),
$$
where $e(s)\in E$.
Recall that $\langle H_{0I},u_+\rangle$ is positive and $\langle H_{0I},u_-\rangle$ is negative. Collecting now all the information we obtain that any nonzero solution of system $\dot z=Bz$ that belong to the invariant cone $C$  asymptotically tends to the line $\mathrm{span}\{(1,u_\pm)\}$ as $s\to\pm\infty$. \\
\begin{center}
\psscalebox{.8 .8} 
{
\begin{pspicture}(0,-2.554099)(5.6482844,2.554099)
\definecolor{colour0}{rgb}{0.8,0.8,0.8}
\pscircle[linecolor=black, linewidth=0.04, fillstyle=solid,fillcolor=colour0, dimen=outer](2.5549216,-0.080781706){2.16}
\psdots[linecolor=black, dotsize=0.18](1.6749215,-0.6407817)
\psdots[linecolor=black, dotsize=0.2](3.4349215,0.57921827)
\psbezier[linecolor=black, linewidth=0.04](1.6749215,-0.6807817)(1.8749216,-1.0407817)(3.0549216,-1.9207817)(3.6549215,-1.860781707763672)
\psbezier[linecolor=black, linewidth=0.04](1.3349216,1.6392183)(0.95492154,1.1592183)(2.4349215,-1.2207817)(4.1149216,-1.5007817077636718)
\psbezier[linecolor=black, linewidth=0.04](1.6949216,1.8792183)(1.4749216,1.5592183)(2.6549215,0.51921827)(3.3949215,0.5992182922363282)
\psline[linecolor=black, linewidth=0.04](0.01492157,-2.5407817)(1.6549215,-0.7007817)
\psline[linecolor=black, linewidth=0.04](3.4749215,0.5992183)(5.6349216,2.5392182)
\psline[linecolor=black, linewidth=0.04](4.4549217,1.8392183)(4.9149218,1.8792183)(4.7949214,1.4392183)
\psline[linecolor=black, linewidth=0.04](0.23492157,-1.8407817)(0.6549216,-1.8007817)(0.57492155,-2.2007818)(0.57492155,-2.2007818)
\rput[bl](0.57492155,-0.5607817){$u_-$}
\rput[bl](3.2749217,0.05921829){$u_+$}
\rput[bl](3.0949216,-1.0007817){$\tilde{u}(s)$}
\psbezier[linecolor=black, linewidth=0.04, linestyle=dashed, dash=0.17638889cm 0.10583334cm](1.6749215,-0.7007817)(1.5349215,-0.8007817)(2.263984,-2.0207818)(2.9749215,-2.2007817077636718)
\psbezier[linecolor=black, linewidth=0.04, linestyle=dashed, dash=0.17638889cm 0.10583334cm](1.0549216,1.4592183)(0.7349216,0.57921827)(2.4549215,-1.3607817)(3.9149215,-1.660781707763672)
\psbezier[linecolor=black, linewidth=0.04, linestyle=dashed, dash=0.17638889cm 0.10583334cm](1.4549216,1.7592183)(1.1749215,1.2592183)(2.9999764,0.11921829)(3.3549216,0.5392182922363281)
\rput[bl](0.05492157,0.2192183){$u(s)$}
\psline[linecolor=black, linewidth=0.04](2.0349216,-1.7207817)(2.2949216,-1.7807817)(2.2949216,-1.5807817)
\psline[linecolor=black, linewidth=0.04](2.7949216,-1.8007817)(3.0349216,-1.7607818)(2.9349215,-1.5207818)
\psline[linecolor=black, linewidth=0.04](1.4949216,0.3392183)(1.7549216,0.2192183)(1.7549216,0.4592183)(1.7749215,0.4192183)
\psline[linecolor=black, linewidth=0.04](1.3149216,-0.060781706)(1.6549215,-0.16078171)(1.6349216,0.11921829)
\psline[linecolor=black, linewidth=0.04](2.2149215,0.5592183)(2.4549215,0.6192183)(2.2949216,0.8592183)
\psline[linecolor=black, linewidth=0.04](2.0149217,1.0592183)(2.2549217,0.9992183)(2.1549215,1.2992183)(2.1549215,1.2992183)
\end{pspicture}
}

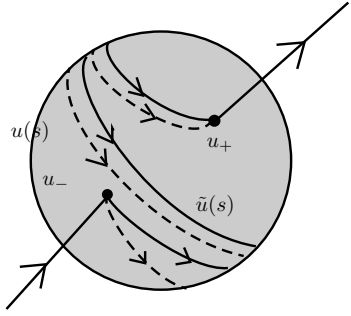
\captionof{figure}{ Two distinct solution $u(s)$ and $\tilde{u}(s)$ of (\ref{777}). }
\end{center}
\end{proof}

Once we have study the system (\ref{11100s}) with rescaled time $s$, we are going to show that the trajectory that we found, which enters in $\bar\lambda_{u^-}$ and goes out from $\bar\lambda_{u^+}$, is an extremal of the system (\ref{11100}) that passes through $\bar{\lambda}$ in finite time.\\

Thus, let us estimate the time $\Delta t$ that this extremal needs to reach $\bar{\lambda}$.

Due to the facts that $\left\langle H_{0I},u_- \right\rangle <0$ and $\left\langle H_{0I}(\lambda),u \right\rangle$ at $\bar{\lambda}_{u_-}$ is continuous with respect to $\lambda_u$, there exist a neighbourhood $O_{\bar{\lambda}_{u-}}$ of $\bar{\lambda}_{u_-}$, in which $\left\langle H_{0I}(\lambda),u \right\rangle$ is bounded from above by a negative constant $c_1<0$, namely ${\left\langle H_{0I}(\lambda),u \right\rangle}_{|O_{\bar{\lambda}_{u-}}}<c_1<0$. \\
Hence, in $O_{\bar{\lambda}_{u-}}$ we have the following estimate of the derivative $\rho'$ $$ \rho'=\rho\, \left\langle H_{0I}(\lambda),u \right\rangle<\rho \,c_1,$$
consequently until $\rho(s)>0$, it holds $$ \int^s_{s_0} \frac{\rho'}{\rho}ds<\int^s_{s_0} c_1 ds,$$
then this inequality implies $\log(\rho(s))<c_1 (s-s_0)+\log(\rho(s_0))$, and so $$\rho(s)<\rho(s_0)e^{c_1 (s-s_0)}.$$
Since $\frac{d}{ds}t(s)=\rho(s)$, the amount of time that we want to estimate is the following
$$\Delta t=\lim_{s\rightarrow \infty}t(s)-t(s_0)=\int^{\infty}_{s_0} \rho(s) ds,
$$
therefore,
$$\Delta t=\int^{\infty}_{s_0} \rho(s) ds<\rho(s_0)\int^{\infty}_{s_0} e^{c_1 (s-s_0)}ds=\frac{\rho(s_0)}{-c_1}<\infty.
$$
The amount of time in which this extremal goes out from $\bar{\lambda}$ may be estimate in an analogous way.
\end{proof}
By the previous Proposition and the fact that every extremal out of $\Lambda$ is smooth, it is proven that there exist a neighbourhood $O_{\bar\lambda}\subset T^*M$ such that for any $z\in O_{\bar\lambda}$ and $\hat{t}>0$ there exists a unique extremal $t\mapsto\lambda(t,z)$ contained in $O_{\bar\lambda}\subset T^*M$ with condition $\lambda(\hat{t},z)=z$.

Let us conclude the proof with the following Proposition.
\begin{proposition}
\label{prop}
The map $(t,z)\rightarrow\lambda(t,z)$ continuously depends on $(t,z)\in I \times O_{\lambda}$.
\end{proposition}
\begin{proof}
At first let us observe that  for all singular point $\lambda\in O_{\bar{\lambda}}$ the phase portrait in the rescaled time after blow up have the same structure. Moreover, the splitting of the phase space on the hyperbolic and central part continuously depend on $\lambda$. This follows from basic facts on invariant submanifold, see \cite{HPS} for details. 

To guarantee continuity of the map $(t,z)\mapsto \lambda(t,z)$ it remains to prove that for each $\varepsilon>0$ there exists a neighbourhood  $O^\varepsilon_{\bar{\lambda}}$ such that the maximum time interval of the extremals in this neighbourhood $\Delta_{O^\varepsilon_{\bar{\lambda}}}t$ is less than $\varepsilon$.

As we saw previously, the solution of (\ref{11100s}) through $\bar{\lambda}$ arrives and goes out at $u_-$ and $u_+$. Let us fix two neighbourhoods $O_{\bar{\lambda}_{u_+}}$ of $\bar{\lambda}_{u_+}$ and $O_{\bar{\lambda}_{u_-}}$ of $\bar{\lambda}_{u_-}$,  we can distinguish three parts of any trajectory close to $\bar{\lambda}$: the parts in $O_{\bar{\lambda}_{u_-}}$ and in $O_{\bar{\lambda}_{u_+}}$, and the part between those neighbourhoods.\\
In this last region,  since each $\rho$-component is close to $0$ and the corresponding time interval with time $s$ is uniformly bounded, as we saw in Proposition \ref{7771}, then $\Delta t$ is arbitrarily small with respect to $O_{\bar{\lambda}}$.\\
Hence, in $O_{\bar{\lambda}_{u_-}}$ we are going to show that there exists a sequence of neighbourhoods of $\bar{\lambda}_{u_-}$
$$\left(O^R_{u_-}\right)_R,$$
such that $$\lim_{R\rightarrow 0^+}\Delta_{O^R_{u_-}}t=0.$$
For simplicity, we are going to prove this fact in $O_{\bar{\lambda}_{u_-}}$, because the situations in $O_{\bar{\lambda}_{u_+}}$ is equivalent.\\
Let us denote $O^R_{u_-}$ a neighbourhood of $\bar{\lambda}_{u_-}$ such that $O^R_{u_-}\subseteq O_{\bar{\lambda}_{u_-}}$, for each $\lambda\in O^R_{u_-}$ $\rho<R$ and $||u-u_-||<R$. Therefore, we can define
$$M_R=\sup_{\lambda\in O^R_{u_-}}\left\langle H_{0I}(\lambda),u \right\rangle,$$
 and assume that it is strictly negative and finite, due to the fact that  we can choose $O_{\bar{\lambda}_{u_-}}$ in which $\left\langle H_{0I}(\lambda),u \right\rangle$ is strictly negative and finite.\\
Hence, for every $\lambda(t(s))$ in $O^R_{u_-}$, until its $\rho$-component is different that zero, it holds
$$\frac{\dot{\rho}(s)}{\rho(s)}<M_R,
$$
then $$\rho(s)<\rho(s_0)e^{M_R(s-s_0)},$$ for every $s>s_0$.\\
Consequently, $\Delta_{O^R_{u_-}}t$ can be estimated in the following way:
$$\Delta_{O^R_{u_-}}t<\int^{\infty}_{s_0}\rho(s_0)e^{M_R(s-s_0)}ds=\frac{\rho(s_0)}{-M_R}<\frac{R}{-M_R}.
$$
Due to the fact that $\lim_{R\rightarrow 0^+}\frac{R}{-M_R}=0$, we have proved that for each $\varepsilon>0$ there exists $O^R_{u_-}$ such that $\Delta_{O^R_{u_-}}t<\varepsilon$. 
\end{proof}
\subsection{Proof of Proposition \ref{809} }
Let us assume that there exist a time-optimal control $\tilde{u}$, and an interval $(\tau_1,\tau_2)$ such that $\tilde{u}$ corresponds to an extremal $\lambda(t)$ in $O_{\bar{\lambda}}$, and $\lambda(t)\in \Lambda$, $\forall t\in(\tau_1,\tau_2)$. By construction, for $t\in(\tau_1,\tau_2)$ it holds
\begin{equation}
\label{556}
\left\lbrace
\begin{array}{l}
\frac{d}{dt}h_1(\lambda(t))=0\\
\vdots\\
\frac{d}{dt}h_k(\lambda(t))=0.
\end{array}
\right.
\end{equation}
Since the maximized Hamiltonian associated with $\tilde{u}$ is
$$\mathcal{H}_{\tilde{u}}(\lambda)=h_0(\lambda)+\tilde{u}_1h_1(\lambda)+\hdots+\tilde{u}_kh_k(\lambda),
$$
by Remark \ref{derivPoiss}, (\ref{556}) implies
\begin{equation*}
H_{0I}(\lambda(t))-H_{IJ}(\lambda(t))\tilde{u}=0.
\end{equation*}
Moreover, due to condition (\ref{cond_eqq}),we can claim that, choosing $O_{\bar{\lambda}}$ small enough, $H_{0I}(\lambda(t))\notin H_{IJ}(\lambda(t))\overline{B^k}$ or $H_{0I}(\lambda(t))\in H_{IJ}(\lambda(t)) B^k$, for all $t\in (\tau_1,\tau_2)$.\\
If $H_{0I}(\lambda(t))\notin H_{IJ}(\lambda(t))\overline{B^k}$, we arrive to a contradiction, because in this case $||\tilde{u}||>1$ but the norm of admissible controls is less equal than $1$. On the other hand, if $H_{0I}(\lambda(t))\in H_{IJ}(\lambda(t)) B^k$, such extremals might exist, but they are not optimal by the Goh Condition, presented at Subsection \ref{subsecgoh}.

\end{document}